\documentclass{amsart}

\usepackage{graphicx,epsfig}

\usepackage{amsmath,amssymb,latexsym, amsfonts, amscd, amsthm, xy}

\usepackage{url}

\usepackage{hyperref}

\input{xy}
\xyoption{all}

\usepackage[usenames]{color}

\makeindex 

=630pt \headheight = 12pt \headsep = 20pt

\hypersetup{
    colorlinks   = true,
      citecolor    = red
}

\hypersetup{linkcolor=red}

\hypersetup{linkcolor=blue}

\theoremstyle{plain}
\newtheorem{theorem}{Theorem}[section]

\newtheorem{corollary}[theorem]{Corollary}

\newtheorem{lemma}[theorem]{Lemma}

\theoremstyle{definition}

\newtheorem{remark}[theorem]{Remark}

\def\bA{\mathbb{A}}

\def\bF{\mathbb{F}}

\def\bZ{\mathbb{Z}}

\def\cG{\mathcal{G}}

\def\cP{\mathcal{P}}

\def\cQ{\mathcal{Q}}

\def\cR{\mathcal{R}}

\def\cS{\mathcal{S}}

\def\fc{\mathfrak{c}}

\def\fq{\mathfrak{q}}

\def\deg{\mathbf{deg}}

\def\card{\mathrm{card}}

\begin{document}

\title{The distribution of the Carlitz binomial coefficients modulo a prime}

\author{Dong Quan Ngoc Nguyen}

\dedicatory{Dedicated to the memory of David Goss}

\date{July 24, 2017}

\address{Department of Applied and Computational Mathematics and Statistics \\
         University of Notre Dame \\
         Notre Dame, Indiana 46556, USA }

\email{\href{mailto:dongquan.ngoc.nguyen@nd.edu}{\tt dongquan.ngoc.nguyen@nd.edu}}

\maketitle

\tableofcontents

\begin{abstract}

For a nonnegative integer $n$, and a prime $\wp$ in $\bF_q[T]$, we prove a result that provides a method for computing the number of integers $m$ with $0 \le m \le n$ for which the Carlitz binomial coefficients $\binom{n}{m}_C$ fall into each of the residue classes modulo $\wp$. Our main result can be viewed as a function field analogue of the Garfield--Wilf theorem.

\end{abstract}

\section{Introduction}

Let $p$ be a prime, and let $q = p^s$ for some $s \in \bZ_{>0}$. Let $\bA = \bF_q[T]$ be the ring of polynomials over the finite field $\bF_q$ of $q$ elements, where $T$ denotes an indeterminate. There are many strong analogues between $\bZ$ and $\bA$ (see Goss \cite{Goss}, Rosen \cite{Rosen}, Thakur \cite{Thakur}, and Weil \cite{Weil}). In this paper, we will further investigate more analogous phenomena between $\bZ$ and $\bA$. 

The factorials $n!$ are a basic object in $\bZ$. In the function field setting, Carlitz \cite{Carlitz} discovered an analogue of the usual factorials which we call the Carlitz factorials throughout this paper. The Carlitz factorials are the images of the integers under the mapping $!_C : \bZ_{\ge 0} \to \bA$ (whose definition will be reviewed in Subsection \ref{ss-Lucas-thm}). For each $n \in \bZ_{\ge 0}$, there is a corresponding Carlitz factorial $n!_C$ which is an element in $\bA$. The Carlitz factorial $n!_C$ is defined using the $q$-adic expansion of $n$, and the elements $[i] = T^{q^i} - T \in \bA$.

Using the function field analogue of the usual factorials, Carlitz proved several beautiful analogues between $\bZ$ and $\bA$, among which is a function field analogue of the Staudt--Clausen theorem (see Carlitz \cite{Carlitz}). 

For each $m, n \in \bZ_{\ge 0}$, the binomial coefficient $\binom{n}{m}$ is defined as $\dfrac{n!}{m!(n - m)!}$ if $n \ge m$, and $0$ if otherwise. For a large integer $n$, it is not easy to understand the divisibility properties of the binomial coefficients $\binom{n}{m}$, where $0 \le m \le n$; for example, can one determine whether the binomial coefficients $\binom{n}{m}$ are divisible by a given prime $p$, or fall into certain residue classes modulo $p$? Motivated by this question, Garfield and Wilf \cite{Garfield-Wilf} studied the distribution of the binomial coefficients modulo a prime $p$. For an integer $n \ge 0$, and a prime $p$ in $\bZ$, the main theorem in Garfield and Wilf \cite{Garfield-Wilf} can tell us the number of integers $m$ with $0 \le m \le n$ for which the binomial coefficients $\binom{n}{m}$ fall into each of the residue classes $1, 2, \ldots, p - 1 \pmod{p}$. The main aim of our paper is to prove a function field analogue of the Garfield-Wilf theorem. 

Following the notion of binomial coefficients in $\bZ$, one can define the Carlitz binomial coefficients in the same way as in $\bZ$ by replacing the usual factorials by the Carlitz ones. More explicitly, the Carlitz binomial coefficient $\binom{n}{m}_C$ is defined as $\dfrac{n!_C}{m!_C(n - m)!_C}$ if $n \ge m$, and $0$ if otherwise. Note that $\binom{n}{m}_C$ is an element in $\bA$ for all integers $n, m$.

Our main result in this paper describes the distribution of the Carlitz binomial coefficients $\binom{n}{m}_C$ modulo a prime $\wp$ in $\bA$ (recall that an element in $\bA$ is called a prime if it is an irreducible polynomial over $\bF_q$.) For an integer $n \ge 0$, and a prime $\wp \in \bA$, our main theorem (see Theorem \ref{main-theorem}) provides a method for computing the number of integers $m$ with $0 \le m \le n$ for which the Carlitz binomial coefficients $\binom{n}{m}_C$ fall into each of the residue classes of $(\bA/\wp\bA)^{\times}$. In order to establish the main result of this paper, we will adapt the proof of the Garfield--Wilf theorem into the function field setting. 

The content of the paper is organized as follows. In Subsection \ref{ss-Lucas-thm}, we recall a function field analogue of the Lucas theorem that is due to Thakur \cite{thakur}. We will need this function field analogue in the proof of one of our main lemmas. In Subsection \ref{subsection-semigroup}, we describe the semigroups that will play a key role in the proof of our main result. In Subsection \ref{subsection-the-distribution-of-the-C-bc}, we state and prove the main result in this paper.

\section{The distribution of the Carlitz binomial coefficients modulo $\wp$}

The aim of this section is to prove our main result about the distribution of the Carlitz binomial coefficients modulo a prime in $\bA$. We begin with Subsection \ref{ss-Lucas-thm} whose main purpose is to recall a function field analogue of the Lucas theorem that is discovered by Thakur \cite{thakur}.

\subsection{A function field analogue of Lucas' theorem}
\label{ss-Lucas-thm}

We begin this subsection by recalling the notion of the Carlitz binomial coefficients that was introduced by Carlitz \cite{Carlitz} (see also Goss \cite{Goss}, and Thakur \cite{Thakur}). 

For each $n \in \bZ_{\ge 0}$, let 
\begin{align*}
n = n_0 + n_1q + \cdots + n_rq^r
\end{align*}
be the $q$-adic expansion of $n$, where $r \in \bZ_{\ge 0}$, and the $n_i$ are integers such that $0 \le n_i < q$. The $n$-th Carlitz factorial, denoted by $n!_C$, is defined by
\begin{align}
\label{C-factorial}
n!_C = \prod_{i = 0}^r D_i^{n_i} \in \bA,
\end{align}
where $D_i = \prod_{r = 0}^{i - 1} (T^{q^i} - T^{q^r})$ for each $i \ge 1$, and $D_0 = 1$. Note that the above equation implies that $0!_C = 1$.

For $\alpha, \beta \in \bZ_{\ge 0}$, the Carlitz binomial coefficient is defined by
\begin{align}
\label{C-binomial}
\binom{\alpha}{\beta}_C = 
\begin{cases}
\dfrac{\alpha!_C}{\beta!_C(\alpha - \beta)!_C} \; \; &\text{if $\alpha \ge \beta$,} \\
0 \; \; &\text{otherwise}.
\end{cases}
\end{align}

We will need the following function field analogue of Lucas' theorem that is due to Thakur \cite{thakur}.

\begin{theorem}
\label{t-thakur}
$(\text{Thakur \cite[Theorem 3.1]{thakur}})$

Let $\wp$ be a prime of degree $h$ in $\bA$. Let $\alpha, \beta \in \bZ_{\ge 0}$. Let $\alpha = \sum \alpha_i q^{hi}$ and $\beta = \sum \beta_i q^{hi}$ be the $q^h$-adic expansions of $\alpha$ and $\beta$, respectively, where $0 \le \alpha_i, \beta_i < q^h$. Then
\begin{align*}
\binom{\alpha}{\beta}_C \equiv \prod \binom{\alpha_i}{\beta_i}_C \pmod{\wp}.
\end{align*}

\end{theorem}

\subsection{The semigroup of all finite words}
\label{subsection-semigroup}

Throughout this subsection, fix an integer $h \in \bZ_{> 0}$. Let $S_h$ be the set defined by 
\begin{align*}
S_h = \{0, 1, 2, \ldots, q^h - 1\},
\end{align*}
and let $S_h^{\omega}$ be the set given by
\begin{align*}
S_h^{\omega} = \bigcup_{i = 1}^{\infty} S_h^i,
\end{align*}
where for each $i \ge 1$,
\begin{align*}
S_h^i = \prod_{j = 1}^{j = i} S_h = \underbrace{S_h \times \cdots \times S_h}_{\text{$i$ copies of $S_h$}}.
\end{align*}

An element $\bar{\alpha}$ belongs to $S_h^{\omega}$ if and only if there exists a unique integer $i \ge 1$ such that $\bar{\alpha} \in S_h^i$, and $\bar{\alpha}$ can be uniquely written in the form 
\begin{align*}
\bar{\alpha} = (\alpha_0, \alpha_1, \ldots, \alpha_{i - 1})
\end{align*}
for some $\alpha_0, \ldots, \alpha_{i - 1} \in S_h$. 
The unique integer $i \ge 1$ for which $\bar{\alpha} \in S_h^i$ is called the \textit{degree of $\bar{\alpha}$}. Equivalently, the degree of $\bar{\alpha}$ is the number of components of $\bar{\alpha}$. In notation, we denote by $\deg(\bar{\alpha})$ the degree of $\bar{\alpha}$.

We define a binary operation $\star : S_h^{\omega} \times S_h^{\omega} \to S_h^{\omega}$ as follows. For each $\bar{\alpha} = (\alpha_0, \ldots, \alpha_r), \bar{\beta} = (\beta_0, \ldots, \beta_s) \in S_h^{\omega}$ for some $r, s \ge 0$, we define $\bar{\alpha} \star \bar{\beta}$ to be the element $(\gamma_0, \ldots, \gamma_{r + s + 1}) \in S_h^{r + s + 2} \subset S_h^{\omega}$, where
\begin{align*}
\gamma_i =
\begin{cases}
\alpha_i \; \; &\text{if $0 \le i \le r$,} \\
\beta_{i - (r + 1)}  \; \; &\text{if $r + 1 \le i \le r + s + 1$.}
\end{cases}
\end{align*}
The set $S_h^{\omega}$ equipped with the binary operation ``$\star$'' is a semigroup. 

\begin{remark}

One can view $S_h$ is an alphabet, and $S_h^{\omega}$ is the set of all finite words over the alphabet $S_h$. The binary operation ``$\star$'' can be seen as concatenation on $S_h^{\omega}$. 

\end{remark}

To each element $\bar{\alpha} = (\alpha_0, \ldots, \alpha_r) \in S_h^{\omega}$, one associates an integer $z_{\bar{\alpha}}(h) \in \bZ_{\ge 0}$ by setting
\begin{align}
\label{z-from-S-h-omega}
z_{\bar{\alpha}}(h) = \sum_{i = 0}^{r} \alpha_i q^{hi}.
\end{align}
Note that $z_{\bar{\alpha}}(h)$ is the $q^h$-adic expansion of the components of $\bar{\alpha}$. 

From (\ref{z-from-S-h-omega}), it is not difficult to see that
\begin{align}
\label{z-alpha*beta}
z_{\bar{\alpha} \star \bar{\beta}}(h) = z_{\bar{\alpha}}(h) + q^{h(\deg(\bar{\alpha}))}z_{\bar{\beta}}(h).
\end{align}

Now let $u \in \bZ_{\ge 0}$, and let $u = \sum_{i = 0}^{\infty} u_iq^{hi}$ be the $q^h$-adic expansion of $u$, where the integers $u_i$ satisfy $0 \le u_i < q^h$, and all but finitely many $u_i$ are zero. 

For each integer $s \in \bZ_{\ge 0}$, we define an integer $u_{(s)} \in \bZ_{\ge 0}$ associated to $u$ by 
\begin{align}
\label{u-s}
u_{(s)}= \sum_{i = 0}^{\infty}  u_{s + i}q^{hi} \in \bZ_{\ge 0}.
\end{align}

For each pair $(r, s) \in \bZ_{\ge 0} \times \bZ_{\ge 0}$, we define an integer $u_{(r, s)} \in \bZ_{\ge 0}$ associated to $u$ by 
\begin{align}
\label{u-r-s}
u_{(r, s)} = \sum_{i = 0}^r u_{s + i}q^{hi} \in \bZ_{\ge 0}.
\end{align}

Note that $u_{(s)}$ and $u_{(r, s)}$ are uniquely determined by $u$, and equations (\ref{u-s}) and (\ref{u-r-s}) are the $q^h$-adic expansions of $u_{(s)}$ and $u_{(r, s)}$, respectively. The next result is straightforward from (\ref{u-s}) and (\ref{u-r-s}).

\begin{lemma}
\label{l1}

Let $u \in \bZ_{\ge 0}$. Then
\begin{align*}
u = u_{(r, 0)} + q^{h(r + 1)}u_{(r + 1)}
\end{align*}
for any integer $r \ge 0$.

\end{lemma}

\begin{lemma}
\label{l2}

Let $\wp$ be a prime of degree $h \ge 1$ in $\bA$. Let $\bar{\alpha}, \bar{\beta} \in S_h^{\omega}$, and let $u \in \bZ_{\ge 0}$. We maintain the notation as above. Then
\begin{align*}
\binom{z_{\bar{\alpha} \star \bar{\beta}}(h)}{u}_C \equiv \binom{z_{\bar{\alpha}}(h)}{u_{(\deg(\bar{\alpha})-1, 0)}}_C \binom{z_{\bar{\beta}}(h)}{u_{(\deg(\bar{\alpha}))}}_C \pmod{\wp},
\end{align*}
where $u_{(\deg(\bar{\alpha}))}, u_{(\deg(\bar{\alpha}) - 1, 0)} \in \bZ_{\ge 0}$ are defined by (\ref{u-s}) and (\ref{u-r-s}), respectively.

\end{lemma}

\begin{proof}

One can write $\bar{\alpha} = (\alpha_0, \ldots, \alpha_r)$, $\bar{\beta} = (\alpha_{r + 1}, \ldots, \alpha_{r + s})$, where $r = \deg(\bar{\alpha}) - 1 \ge 0$, $s = \deg(\bar{\beta})$, and the $\alpha_i$ are in $S_h$. From (\ref{z-from-S-h-omega}), one can write $z_{\bar{\alpha} \star \bar{\beta}}(h)$ in the form
\begin{align}
\label{eq1-l2}
z_{\bar{\alpha} \star \bar{\beta}}(h) = \sum_{i = 0}^{\infty} \alpha_i q^{hi},
\end{align}
where we let $\alpha_i = 0$ for all $i > r + s$. Note that (\ref{eq1-l2}) is the $q^h$-adic expansion of $z_{\bar{\alpha} \star \bar{\beta}}(h)$.

Let $u = \sum_{i = 0}^{\infty} \gamma_iq^{hi}$ be the $q^h$-adic expansion of $u$, where the integers $\gamma_i$ satisfy $0 \le \gamma_i < q^h$, and all but finitely many $\gamma_i$ are zero. From (\ref{eq1-l2}) and Theorem \ref{t-thakur}, we deduce that
\begin{align}
\label{eq2-l2}
\binom{z_{\bar{\alpha} \star \bar{\beta}}(h)}{u}_C \equiv \prod_{i = 0}^{\infty}\binom{\alpha_i}{\gamma_i}_C = \left( \prod_{i = 0}^r\binom{\alpha_i}{\gamma_i}_C\right) \left( \prod_{j = r + 1}^{\infty}\binom{\alpha_j}{\gamma_j}_C\right)\pmod{\wp}.
\end{align}

By (\ref{u-r-s}) and (\ref{z-from-S-h-omega}), one sees that
\begin{align*}
u_{(\deg(\bar{\alpha}) - 1, 0)} = u_{(r, 0)} = \sum_{i = 0}^r\gamma_iq^{hi},
\end{align*}
and
\begin{align*}
z_{\bar{\alpha}}(h) = \sum_{i = 0}^r\alpha_iq^{hi}.
\end{align*}
Hence it follows from Theorem \ref{t-thakur} that
\begin{align}
\label{eq3-l2}
\binom{z_{\bar{\alpha}}(h)}{u_{(\deg(\bar{\alpha})-1, 0)}}_C \equiv  \prod_{i = 0}^r\binom{\alpha_i}{\gamma_i}_C \pmod{\wp}.
\end{align}

Using (\ref{u-s}), one sees that
\begin{align*}
u_{(\deg(\bar{\alpha}))} = u_{(r + 1)} = \sum_{j = 0}^{\infty} \gamma_{r + 1 + j}q^{hj}.
\end{align*}
Since $\bar{\beta} = (\alpha_{r + 1}, \ldots, \alpha_{r + s})$, and $\alpha_j = 0$ for all $j > r + s$, one can write $z_{\bar{\beta}}(h)$ in the form
\begin{align*}
z_{\bar{\beta}}(h) = \sum_{j = 0}^{\infty}\alpha_{r + 1+ j}q^{hj}.
\end{align*}
Hence Theorem \ref{t-thakur} implies that
\begin{align}
\label{eq4-l2}
\binom{z_{\bar{\beta}}(h)}{u_{(\deg(\bar{\alpha})-1, 0)}}_C \equiv  \prod_{j = 0}^{\infty}\binom{\alpha_{r + 1+ j}}{\gamma_{r + 1 + j}}_C = \prod_{j = r + 1}^{\infty}\binom{\alpha_{j}}{\gamma_{j}}_C \pmod{\wp}.
\end{align}

It follows from (\ref{eq2-l2}), (\ref{eq3-l2}), (\ref{eq4-l2}) that 
\begin{align*}
\binom{z_{\bar{\alpha} \star \bar{\beta}}(h)}{u}_C \equiv \binom{z_{\bar{\alpha}}(h)}{u_{(\deg(\bar{\alpha})-1, 0)}}_C \binom{z_{\bar{\beta}}(h)}{u_{(\deg(\bar{\alpha}))}}_C \pmod{\wp},
\end{align*}
which proves our contention.

\end{proof}

The next result is a special case of Lemma \ref{l2}.

\begin{corollary}
\label{c1}

Let $\wp$ be a prime of degree $h \ge 1$ in $\bA$. Let $\bar{\alpha}, \bar{\beta} \in S_h^{\omega}$, and let $u, v$ be integers such that $0 \le u \le z_{\bar{\alpha}}(h)$ and $0 \le v \le z_{\bar{\beta}}(h)$. Then
\begin{align*}
\binom{ z_{\bar{\alpha}\star \bar{\beta}}(h)}{u + q^{h\deg(\bar{\alpha})}v}_C \equiv \binom{z_{\bar{\alpha}}(h)}{u}_C \binom{z_{\bar{\beta}}(h)}{v}_C \pmod{\wp}.
\end{align*}

\end{corollary}

\begin{proof}

Set
\begin{align*}
w = u + q^{h\deg(\bar{\alpha})}v \in \bZ_{\ge 0}.
\end{align*}
If $u = 0$, then since  $q^{h\deg(\bar{\alpha})}v \equiv 0 \pmod{q^{h\deg(\bar{\alpha})}}$, we deduce from  (\ref{u-r-s}) that
\begin{align}
\label{e1-c1}
w_{(\deg(\bar{\alpha}) - 1, 0)} = 0 = u.
\end{align}

Suppose now that $u > 0$, and let $u = u_0 + u_1q^h + \ldots + u_sq^{hs}$ be the $q^h$-adic expansion of $u$, where $u_s \ne 0$, and $0\le u_i \le q^h - 1$ for all $0 \le i \le s$. Let $\bar{\alpha} = (\alpha_0, \ldots, \alpha_r)$, where $r = \deg(\bar{\alpha}) - 1$, and the $\alpha_i$ are in $S_h$. Since $u \le z_{\bar{\alpha}}(h)$, there exists some $0 \le i \le r$ such that $\alpha_i \ne 0$. Let $e$ be the largest integer such that $\alpha_e > 0$, and $\alpha_i = 0$ for all $e< i \le r$. Then $z_{\bar{\alpha}}(h)$ can be written in the form
\begin{align*}
z_{\bar{\alpha}}(h) = \alpha_0 + \alpha_1q^h + \ldots + \alpha_eq^{he}.
\end{align*}
We contend that $s \le e$; otherwise, $s > e$, and hence $s \ge e + 1$. Since $0 \le \alpha_j \le q^h - 1$, we obtain that
\begin{align*}
z_{\bar{\alpha}}(h) = \alpha_0 + \alpha_1q^h + \ldots + \alpha_eq^{he} &\le (q^h - 1)(1 + q^h + \ldots + q^{he}) \\
&= q^{h(e + 1)} - 1 \\
&\le q^{hs} - 1 \\
&< q^{hs} \\
&\le u_0 + u_1q^h + \ldots u_sq^{hs} = u,
\end{align*}
which is a contradiction to the assumption that $z_{\bar{\alpha}}(h) \ge u$. Thus $s \le e$. In particular this implies that 
\begin{align*}
s \le r = \deg(\bar{\alpha}) - 1 < \deg(\bar{\alpha}),
\end{align*}
which in turn implies that
\begin{align*}
u = u_0 + u_1q^h + \ldots + u_sq^{hs} &\le (q^h - 1)(1 + q^h + \ldots + q^{hs}) \\
&= q^{h(s + 1)} - 1 \\
&\le q^{h \deg(\bar{\alpha})} - 1 \\
&< q^{h \deg(\bar{\alpha})}.
\end{align*}

Since $u > 0$, we deduce from the above inequality that
\begin{align*}
u = u_0 + u_1q^h + \ldots + u_sq^{hs} \not\equiv 0 \pmod{q^{h\deg(\bar{\alpha})}}
\end{align*}
Since $q^{h\deg(\bar{\alpha})}v \equiv 0 \pmod{q^{h\deg(\bar{\alpha})}}$, we deduce from the above congruence and (\ref{u-r-s}) that
\begin{align}
\label{e2-c1}
w_{(\deg(\bar{\alpha}) - 1, 0)} = u.
\end{align}
By (\ref{e1-c1}) and (\ref{e2-c1}), we deduce that
\begin{align}
\label{e3-c1}
w_{(\deg(\bar{\alpha}) - 1, 0)} = u.
\end{align}

Now write $v = \sum_{j = 0}^{\infty} v_j q^{hj}$, where $0 \le v_j < q^h$ for $j \ge 0$. Then
\begin{align*}
w = u + q^{h\deg(\bar{\alpha})}v = \sum_{j = 0}^su_jq^{hj} + \sum_{j = 0}^{\infty} v_jq^{h(j + \deg(\bar{\alpha}))}
\end{align*}
is the $q^h$-adic expansion of $w$. Hence we deduce from  (\ref{u-s}) that 
\begin{align}
\label{e4-c1}
w_{(\deg(\bar{\alpha}))} =  \sum_{j = 0}^{\infty} v_j q^{hj} = v.
\end{align}

Applying Lemma \ref{l2} with $w$ in the role of $u$, we deduce from (\ref{e3-c1}) and (\ref{e4-c1}) that
\begin{align*}
\binom{z_{\bar{\alpha} \star \bar{\beta}}(h)}{u + q^{h\deg(\bar{\alpha})}v}_C =\binom{z_{\bar{\alpha} \star \bar{\beta}}(h)}{w}_C \equiv \binom{z_{\bar{\alpha}}(h)}{w_{(\deg(\bar{\alpha})-1, 0)}}_C \binom{z_{\bar{\beta}}(h)}{w_{(\deg(\bar{\alpha}))}}_C = \binom{z_{\bar{\alpha}}(h)}{u}_C \binom{z_{\bar{\beta}}(h)}{v}_C \pmod{\wp}.
\end{align*}

\end{proof}

\begin{remark}
\label{r-c1}

Let $\bar{\alpha}, \bar{\beta} \in S_h^{\omega}$. Let 
\begin{align}
\label{e1-r-c1}
w = u + q^{h\deg(\bar{\alpha})}v
\end{align}
for some integers $u, v$ with $0 \le u \le z_{\bar{\alpha}}(h)$, and $0 \le v \le z_{\bar{\beta}}(h)$. Then $w$ is uniquely written in the form of (\ref{e1-r-c1}), i.e., if $w = u_1 + q^{h\deg(\bar{\alpha})}v_1$ for some integers $u_1, v_1$ with $0 \le u_1 \le z_{\bar{\alpha}}(h)$ and $0 \le v_1 \le z_{\bar{\beta}}(h)$, then $ u = u_1$ and $v = v_1$. Assume the contrary, i.e., there exist $0 \le u, u_1 \le z_{\bar{\alpha}}(h)$, and $0 \le v, v_1 \le z_{\bar{\beta}}(h)$ such that $(u, v) \ne (u_1, v_1)$ and 
\begin{align*}
w = u + q^{h\deg(\bar{\alpha})}v = u_1 + q^{h\deg(\bar{\alpha})}v_1.
\end{align*}
Hence 
\begin{align}
\label{e2-r-c1}
u - u_1 = q^{h\deg(\bar{\alpha})}(v_1 - v).
\end{align}
We see from the above equation that $u = u_1$ if and only if $v = v_1$. Since $(u, v) \ne (u_1, v_1)$, without loss of generality, we can assume that $u - u_1> 0$. Let $u - u_1 = \sum_{i = 0}^r\epsilon_i q^{hi}$ be the $q^h$-adic expansion of $u - u_1$, where $\epsilon_r > 0$, and $0 \le \epsilon_i \le q^h - 1$ for all $0 \le i \le r$. Note that $0 < u - u_1 \le z_{\bar{\alpha}}(h)$. Hence using the same arguments as in the proof of Corollary \ref{c1}, one can show that $r < \deg(\bar{\alpha})$, and thus $0 < u - u_1 < q^{h\deg(\bar{\alpha})}$. Therefore
\begin{align*}
u - u_1 \not\equiv 0 \pmod{q^{h\deg(\bar{\alpha})}},
\end{align*}
which is a contradiction to (\ref{e2-r-c1}). Hence $u = u_1$ and $v = v_1$, which proves our contention.

\end{remark}

Let $\wp$ be a prime of degree $h \ge 1$ in $\bA$, and let $\fq$ be a primitive root modulo $\wp$. For each $\bar{\alpha} \in S_h^{\omega}$ and each $j \in \bZ$, let $\cS_j(\bar{\alpha})$ be the subset of $\bZ$ defined by
\begin{align}
\label{cS-j-bar-alpha}
\cS_j(\bar{\alpha}) = \{u \in \bZ \; |\;  \text{$ 0 \le u \le z_{\bar{\alpha}}(h)$ and $\binom{z_{\bar{\alpha}}(h)}{u}_C \equiv \fq^j \pmod{\wp}$}  \}.
\end{align}

The next lemma will play a key role in the proof of our main theorem. 

\begin{lemma}
\label{main-l}

Let $\wp$ be a prime of degree $h \ge 1$ in $\bA$, and let $\fq$  be a primitive root modulo $\wp$. Let $m$ be an integer, and let $\bar{\alpha}, \bar{\beta} \in S_h^{\omega}$. Then the set $\cup_{j = 0}^{q^h - 2}\cS_j(\bar{\alpha}) \times \cS_{m - j}(\bar{\beta})$ is in bijection with $\cS_m(\bar{\alpha}\star \bar{\beta})$.

\end{lemma}

\begin{proof}

We define a mapping $\Psi : \cup_{j = 0}^{q^h - 2}\cS_j(\bar{\alpha}) \times \cS_{m - j}(\bar{\beta}) \to \cS_m(\bar{\alpha}\star \bar{\beta})$ as follows. Let $(u, v) \in \cS_j(\bar{\alpha}) \times \cS_{m - j}(\bar{\beta})$ for some $0 \le j \le q^h - 2$. By Corollary \ref{c1}, we see that
\begin{align}
\label{e1-main-l}
\binom{z_{\bar{\alpha}*\bar{\beta}}(h)}{u + q^{h\deg(\bar{\alpha})}v}_C \equiv \binom{z_{\bar{\alpha}}(h)}{u}_C\binom{z_{\bar{\beta}}(h)}{u}_C \equiv \fq^j\fq^{m - j} = \fq^m  \pmod{\wp}.
\end{align}

Since $0 \le u \le z_{\bar{\alpha}}(h)$ and $0 \le v \le z_{\bar{\beta}}(h)$, we deduce from (\ref{z-alpha*beta}) that
\begin{align}
\label{e2-main-l}
0 \le u + q^{h\deg(\bar{\alpha})}v \le z_{\bar{\alpha}*\bar{\beta}}(h).
\end{align}

Set $\Psi(u, v) = u + q^{h\deg(\bar{\alpha})}v$. We see from (\ref{e1-main-l}) and (\ref{e2-main-l}) that $\Psi(u, v) \in \cS_m(\bar{\alpha}\star \bar{\beta})$, and thus $\Psi$ is a well-defined mapping from $\cup_{j = 0}^{q^h - 2}\cS_j(\bar{\alpha}) \times \cS_{m - j}(\bar{\beta})$ to $\cS_m(\bar{\alpha}\star \bar{\beta})$. We contend that $\Psi$ is a bijection. 

We first prove that $\Psi$ is surjective. Indeed, let $w$ be any integer in $\cS_m(\bar{\alpha}\star \bar{\beta})$. Set 
\begin{align}
\label{u-main-l}
u = w_{(\deg(\bar{\alpha})-1, 0)} \ge 0,
\end{align}
and
\begin{align}
\label{v-main-l}
v = w_{(\deg(\bar{\alpha}))} \ge 0.
\end{align}
Using Lemma \ref{l2}, we see that
\begin{align}
\label{e3-main-l}
\binom{z_{\bar{\alpha}}(h)}{u}_C \binom{z_{\bar{\beta}}(h)}{v}_C = \binom{z_{\bar{\alpha}}(h)}{w_{(\deg(\bar{\alpha})-1, 0)}}_C \binom{z_{\bar{\beta}}(h)}{w_{(\deg(\bar{\alpha}))}}_C \equiv \binom{z_{\bar{\alpha} \star \bar{\beta}}(h)}{w}_C \equiv \fq^m \pmod{\wp},
\end{align}
which in particular implies that $\binom{z_{\bar{\alpha}}(h)}{u}_C, \binom{z_{\bar{\beta}}(h)}{v}_C$ belong to the multiplicative group $(\bA/\wp\bA)^{\times}$. Since $\fq$ is a primitive root modulo $\wp$, we see that
\begin{align}
\label{e4-main-l}
\binom{z_{\bar{\alpha}}(h)}{u}_C \equiv \fq^j \pmod{\wp}
\end{align}
for some $0 \le j \le q^h - 2$. 

Since $\binom{z_{\bar{\alpha}}(h)}{u}_C \ne 0$ and $\binom{z_{\bar{\beta}}(h)}{v}_C \ne 0$, the definition of the Carlitz binomial coefficients implies that $0 \le u \le z_{\bar{\alpha}}(h)$ and $0 \le v \le z_{\bar{\beta}}(h)$. It follows from (\ref{e3-main-l}) and (\ref{e4-main-l}) that
\begin{align}
\label{e5-main-l}
\binom{z_{\bar{\beta}}(h)}{v}_C \equiv \fq^{m - j} \pmod{\wp}.
\end{align}
By (\ref{e4-main-l}) and (\ref{e5-main-l}), we deduce that $u \in \cS_j(\bar{\alpha})$ and $v \in \cS_{m - j}(\bar{\beta})$. It follows from Lemma \ref{l1}, (\ref{u-main-l}), and (\ref{v-main-l}) that
\begin{align*}
\Psi(u, v) = u + q^{h\deg(\bar{\alpha})}v = w_{(\deg(\bar{\alpha})-1, 0)} + q^{h\deg(\bar{\alpha})}w_{(\deg(\bar{\alpha}))} = w,
\end{align*}
which proves that $\Psi$ is surjective. 

We now prove that $\Psi$ is injective. Assume that $(u, v), (u_1, v_1) \in \cup_{j = 0}^{q^h - 2}\cS_j(\bar{\alpha}) \times \cS_{m - j}(\bar{\beta})$ such that 
\begin{align*}
u + q^{h\deg(\bar{\alpha})}v = \Psi(u, v) = \Psi(u_1, v_1) = u_1 + q^{h\deg(\bar{\alpha})}v_1.
\end{align*}
Remark \ref{r-c1} immediately implies that $(u, v) = (u_1, v_1)$, which proves that $\Psi$ is injective.

From what we have proved above, $\Psi$ is a bijection.

\end{proof}

\begin{corollary}
\label{c2}

Let $\wp$ be a prime of degree $h \ge 1$ in $\bA$, and let $\fq$  be a primitive root modulo $\wp$. Let $\bar{\alpha}, \bar{\beta} \in S_h^{\omega}$. Then
\begin{align*}
\card(\cS_m(\bar{\alpha}\star \bar{\beta})) = \sum_{j = 0}^{q^h - 2}\card(\cS_j(\bar{\alpha}))\card(\cS_{m - j}(\bar{\beta}))
\end{align*}
for any $m \in \bZ$, where $\card(\cdot)$ denotes the cardinality of a set.

\end{corollary}

\begin{proof}

Take any integer $m$. By Lemma \ref{main-l}, we know that the set $\cup_{j = 0}^{q^h - 2}\cS_j(\bar{\alpha}) \times \cS_{m - j}(\bar{\beta})$ is in bijection with $\cS_m(\bar{\alpha}\star \bar{\beta})$, and thus
\begin{align*}
\card(\cS_m(\bar{\alpha}\star \bar{\beta})) = \card\left(\cup_{j = 0}^{q^h - 2}\cS_j(\bar{\alpha}) \times \cS_{m - j}(\bar{\beta})\right).
\end{align*}

On the other hand, we see that
\begin{align*}
\left(\cS_j(\bar{\alpha}) \times \cS_{m - j}(\bar{\beta})\right) \bigcap \left(\cS_i(\bar{\alpha}) \times \cS_{m - i}(\bar{\beta})\right) =\emptyset
\end{align*}
for any $0 \le i \ne j \le q^h - 2$. Hence
\begin{align*}
\card(\cS_m(\bar{\alpha}\star \bar{\beta})) =\card\left(\cup_{j = 0}^{q^h - 2}\cS_j(\bar{\alpha}) \times \cS_{m - j}(\bar{\beta})\right) = \sum_{j = 0}^{q^h - 2}\card(\cS_j(\bar{\alpha}))\card(\cS_{m - j}(\bar{\beta})),
\end{align*}
which proves our contention.

\end{proof}

\subsection{The distribution of the Carlitz binomial coefficients modulo $\wp$}
\label{subsection-the-distribution-of-the-C-bc}

In this subsection, we prove the main result in this paper. We begin by introducing some notation that will be used in the statement of our main theorem in this paper. Throughout this subsection, we fix a prime $\wp$ of degree $h \ge 1$ in $\bA$, and let $\fq$ be a primitive root modulo $\wp$. 

We define a sequence of polynomials $\{\cG_n(x)\}_{n \ge 0} \subset \bZ[x]$ as follows. Take an arbitrary integer $n \in \bZ_{\ge 0}$. For each integer $j$, we define $\epsilon_j(n)$ to be the number of integers $m$ such that $0 \le m \le n$ and the Carlitz binomial coefficient $\binom{n}{m}_C$ satisfies
\begin{align}
\label{eq-fc-n}
\binom{n}{m}_C \equiv \fq^j \pmod{\wp}.
\end{align}
In other words,
\begin{align*}
\epsilon_j(n) = \card\left(\left\{m \in \bZ \; | \; \text{$0 \le m \le n$ and $\binom{n}{m}_C \equiv \fq^j \pmod{\wp}$} \right\} \right). 
\end{align*}

Let $\cG_n(x) \in \bZ[x]$ be the polynomial defined by
\begin{align}
\label{eq-cG_n}
\cG_n(x) = \sum_{i = 0}^{q^h - 2}\epsilon_i(n)x^i.
\end{align}

\begin{remark}
\label{G0-remark}

From the definition of $\epsilon_j(n)$, we see that
\begin{align*}
\epsilon_j(0) =
\begin{cases}
1 \; \; &\text{if $j = 0$,} \\
0 \; \; &\text{if $1 \le j \le  q^h - 2$.} 
\end{cases}
\end{align*}
It follows that 
\begin{align*}
\cG_0(x) = 1.
\end{align*}

\end{remark}

\begin{remark}
\label{r-in-main-thm}

Fix an integer $n \in \bZ_{\ge 0}$, and let $\epsilon_j(n)$ be as above for each integer $j$. Since the multiplicative group $(\bA/\wp\bA)^{\times}$ is of cardinality $q^h - 1$, it follows from (\ref{eq-fc-n}) that $\epsilon_j(n)$ is periodic in $j$ with the period $q^h - 1$, i.e.,
\begin{align*}
\epsilon_j(n) = \epsilon_{j + \ell(q^h - 1)}(n)
\end{align*}
for any $\ell \in \bZ$.

\end{remark}

For each $n \in \bZ_{\ge 0}$, we define a set of non-negative integers $\{\fc_j(n)\}_{j = 0}^{q^h - 1}$ associated to $n$ as follows. If $n = 0$, set 
\begin{align*}
\fc_0(0) = 1,
\end{align*}
and
\begin{align*}
\fc_j(0) = 0
\end{align*}
for all $1 \le j \le q^h - 1$. 

If $n > 0$, let $n = \sum_{i = 0}^r n_i q^{hi}$ be the $q^h$-adic expansion of $n$, where $0 \le n_i \le q^h - 1$ for all $0 \le i \le r$, and $n_r > 0$. For each integer $0 \le j \le q^h - 1$, let $\fc_j(n)$ be the number of times the integer $j$ appears in the set $\{n_0, n_1, \ldots, n_r\}$, where the latter is the set of digits occurring in the $q^h$-expansion of $n$. 

\begin{remark}
\label{cj-remark}

If we make a convention that $0$ appears only one time in the $q^h$-adic expansion of the integer $n = 0$, then $\fc_j(0)$ can be interpreted as the number of times the integer $j$ appears in the set $\{0\}$ for each $0 \le j \le q^h - 1$.

\end{remark}

We are now ready to state our main result in this paper which can be viewed as a function field analogue of the Garfield--Wilf theorem.

\begin{theorem}
\label{main-theorem}

Let $\wp$ be a prime of degree $h \ge 1$ in $\bA$, and let $n \in \bZ_{\ge 0}$. Then
\begin{align*}
\cG_n(x) \equiv \prod_{i = 0}^{q^h - 1}\cG_i(x)^{\fc_i(n)} \pmod{(x^{q^h - 1} - 1)}.
\end{align*}

\end{theorem}

\begin{proof}

Throughout the proof, we will use the same notation as in Subsection \ref{subsection-semigroup}. Let $S_h^{\omega}$ be the semigroup defined in Subsection \ref{subsection-semigroup}. Let $W_h$ be the semigroup $\bZ[x]/(x^{q^h - 1} - 1)\bZ[x]$ equipped with the multiplication of polynomials modulo $(x^{q^h - 1} - 1)$. Every element in $W_h$ can be represented by a polynomial $q(x) \in \bZ[x]$ of degree $\le q^h - 2$. 

Let $\Gamma: S_h^{\omega} \rightarrow W_h$ be the mapping defined by
\begin{align}
\label{Gamma-d}
\Gamma(\bar{\alpha}) = \sum_{m = 0}^{q^h - 2} \epsilon_m(z_{\bar{\alpha}}(h))x^m
\end{align}
for each $\bar{\alpha} \in S_h^{\omega}$, where $z_{\bar{\alpha}}(h)$ is defined by (\ref{z-from-S-h-omega}). We contend that $\Gamma$ is a semigroup homomorphism, i.e., $\Gamma(\bar{\alpha}\star \bar{\beta}) = \Gamma(\bar{\alpha})\Gamma(\bar{\beta})$ for any $\bar{\alpha}, \bar{\beta} \in S_h^{\omega}$. Indeed we see that
\begin{align*}
\Gamma(\bar{\alpha}\star \bar{\beta}) &= \sum_{m = 0}^{q^h - 2} \epsilon_m(z_{\bar{\alpha}\star \bar{\beta}}(h))x^m \\
&=  \sum_{m = 0}^{q^h - 2} \card(\cS_m(\bar{\alpha}\star \bar{\beta}))x^m,
\end{align*}
where $\cS_m(\bar{\alpha}\star \bar{\beta})$ is defined by (\ref{cS-j-bar-alpha}) for each $0 \le m \le q^h - 2$, and $\card(\cdot)$ denotes the cardinality of a set. 

By Corollary \ref{c2}, we deduce that
\begin{align}
\label{e1-main-t}
\Gamma(\bar{\alpha}\star \bar{\beta}) &= \sum_{m = 0}^{q^h - 2} \sum_{j = 0}^{q^h - 2}\card(\cS_j(\bar{\alpha}))\card(\cS_{m - j}(\bar{\beta}))x^m \nonumber \\
&=  \sum_{m = 0}^{q^h - 2}\sum_{j = 0}^{q^h - 2} \epsilon_j(z_{\bar{\alpha}}(h)) \epsilon_{m - j}(z_{\bar{\beta}}(h))x^m.
\end{align}

By Remark \ref{r-in-main-thm}, we know that 
\begin{align}
\label{e2-main-t}
\epsilon_{m - j}(z_{\bar{\beta}}(h)) = \epsilon_{q^h - 1+ (m - j)}(z_{\bar{\beta}}(h)).
\end{align}

Since 
\begin{align*}
x^{q^h - 1 + m} \equiv x^m \pmod{(x^{q^h - 1} -1)}
\end{align*}
for any $0 \le m \le q^h - 2$, we deduce from (\ref{e1-main-t}), (\ref{e2-main-t}) that 
\begin{align}
\label{e3-main-t}
\Gamma(\bar{\alpha}\star \bar{\beta}) &= \sum_{m = 0}^{q^h - 2}\sum_{j = 0}^{q^h - 2} \epsilon_j(z_{\bar{\alpha}}(h)) \epsilon_{q^h - 1 + m - j}(z_{\bar{\beta}})x^{q^h - 1 + m} \pmod{(x^{q^h - 1} -1)} \nonumber \\
&= \sum_{j = 0}^{q^h - 2} \epsilon_j(z_{\bar{\alpha}}(h)) x^j\left(\sum_{m = 0}^{q^h - 2}  \epsilon_{q^h - 1 + m - j}(z_{\bar{\beta}}(h))x^{q^h - 1 + m - j}\right) \pmod{(x^{q^h - 1} -1)} \nonumber\\
&=  \sum_{j = 0}^{q^h - 2} \epsilon_j(z_{\bar{\alpha}}(h)) x^j\Lambda(j)  \pmod{(x^{q^h - 1} -1)},
\end{align}
where
\begin{align}
\label{Lambda-main-t}
\Lambda(j) =  \sum_{m = 0}^{q^h - 2}  \epsilon_{q^h - 1 + m - j}(z_{\bar{\beta}}(h))x^{q^h - 1 + m - j}.
\end{align}
We contend that
\begin{align}
\label{e4-main-t}
\Lambda(j) \equiv \Lambda(0)  \pmod{(x^{q^h - 1} -1)}
\end{align}
for all $0 \le j \le q^h - 2$. 

In order to prove (\ref{e4-main-t}), it suffices to prove that $\Lambda(j) \equiv \Lambda(j + 1)  \pmod{(x^{q^h - 1} -1)}$ for all $0 \le j \le q^h - 3$. Indeed, take an arbitrary integer $0 \le j \le q^h - 3$. We see that
\begin{align}
\label{e5-main-t}
\Lambda(j + 1) &=  \sum_{m = 0}^{q^h - 2}  \epsilon_{q^h - 1 + m - (j + 1)}(z_{\bar{\beta}}(h))x^{q^h - 1 + m - (j + 1)} \nonumber \\
&=   \epsilon_{q^h - 2 - j}(z_{\bar{\beta}}(h))x^{q^h - 2 - j} + \sum_{m = 1}^{q^h - 2}  \epsilon_{q^h - 1 + (m - 1) - j}(z_{\bar{\beta}}(h))x^{q^h - 1 + (m - 1) - j} \nonumber \\
&= \epsilon_{q^h - 2 - j}(z_{\bar{\beta}}(h))x^{q^h - 2 - j} + \sum_{m = 0}^{q^h - 3}  \epsilon_{q^h - 1 + m  - j}(z_{\bar{\beta}}(h))x^{q^h - 1 + m - j}.
\end{align}
By Remark \ref{r-in-main-thm}, and since $x^{q^h - 2 - j} \equiv x^{2q^h - 3 - j} \pmod{(x^{q^h - 1} - 1)}$, we deduce that 
\begin{align}
\label{e6-main-t}
 \epsilon_{q^h - 2 - j}(z_{\bar{\beta}}(h))x^{q^h - 2 - j} \equiv \epsilon_{2q^h - 3 - j}(z_{\bar{\beta}}(h))x^{2q^h - 3 - j} = \epsilon_{q^h - 1 + (q^h - 2) - j}(z_{\bar{\beta}}(h))x^{q^h - 1 + (q^h - 2) - j} \pmod{(x^{q^h - 1} - 1)}.
\end{align}

By (\ref{e5-main-t}) and (\ref{e6-main-t}), we deduce that
\begin{align*}
\Lambda(j + 1) &\equiv \epsilon_{q^h - 1 + (q^h - 2) - j}(z_{\bar{\beta}}(h))x^{q^h - 1 + (q^h - 2) - j}  + \sum_{m = 0}^{q^h - 3}  \epsilon_{q^h - 1 + m  - j}(z_{\bar{\beta}})x^{q^h - 1 + m - j} \pmod{(x^{q^h - 1} - 1)} \\
&= \sum_{m = 0}^{q^h - 2 }  \epsilon_{q^h - 1 + m  - j}(z_{\bar{\beta}}(h))x^{q^h - 1 + m - j} \pmod{(x^{q^h - 1} - 1)} \\
&= \Lambda(j) \pmod{(x^{q^h - 1} - 1)},
\end{align*}
and thus (\ref{e4-main-t}) follows immediately.

On the other hand, by Remark \ref{r-in-main-thm}, and since $x^{q^h - 1 + m} \equiv x^{m} \pmod{(x^{q^h - 1} - 1)}$ for each $0 \le m \le q^h - 2$, we see that
\begin{align}
\label{e7-main-t}
\Lambda(0) =  \sum_{m = 0}^{q^h - 2}  \epsilon_{q^h - 1 + m}(z_{\bar{\beta}}(h))x^{q^h - 1 + m} \equiv  \sum_{m = 0}^{q^h - 2}  \epsilon_{m}(z_{\bar{\beta}}(h))x^{m} = \Gamma(\bar{\beta}) \pmod{(x^{q^h - 1} - 1)}.
\end{align}

By (\ref{e3-main-t}), (\ref{e4-main-t}), (\ref{e7-main-t}), we deduce that
\begin{align*} 
\Gamma(\bar{\alpha}\star \bar{\beta}) \equiv  \sum_{j = 0}^{q^h - 2} \epsilon_j(z_{\bar{\alpha}}(h)) x^j\Lambda(0) \equiv \left( \sum_{j = 0}^{q^h - 2} \epsilon_j(z_{\bar{\alpha}}(h)) x^j \right) \Gamma(\bar{\beta}) = \Gamma(\bar{\alpha})\Gamma(\bar{\beta}) \pmod{(x^{q^h - 1} - 1)},
\end{align*}
which proves that $\Gamma$ is a semigroup homomorphism.

Now we interpret the polynomial $\cG_n(x)$ in terms of the semigroup homomorphism $\Gamma$. If $n = 0$, then Theorem \ref{main-theorem} follows immediately from Remarks \ref{G0-remark} and \ref{cj-remark}.

If $n > 0$, let $n = \sum_{i = 0}^rn_i q^{hi}$ be the $q^h$-adic expansion of $n$, where $0 \le n_i \le q^h - 1$ for each $0 \le i \le r$, and $n_r > 0$. Set
\begin{align*}
\bar{\alpha} = (n_0, n_1, \ldots, n_r) \in S_h^{\omega},
\end{align*}
and $\bar{\beta}_i = n_i \in S_h^{\omega}$ for each $0 \le i \le r$. We see that 
\begin{align}
\label{e11-main-t}
\bar{\alpha} = \bar{\beta}_0\star \bar{\beta}_1 \star \cdots \star \bar{\beta}_r.
\end{align}

On the other hand, (\ref{z-from-S-h-omega}) implies that $n = z_{\bar{\alpha}}(h)$, and $n_i = z_{\bar{\beta}_i}(h)$ for each $0 \le i \le r$. Thus
\begin{align}
\label{e12-main-t}
\cG_n(x) = \sum_{j = 0}^{q^h - 2} \epsilon_j(n)x^j = \sum_{j = 0}^{q^h - 2} \epsilon_j(z_{\bar{\alpha}}(h))x^j = \Gamma(\bar{\alpha}) \pmod{(x^{q^h - 1} - 1)},
\end{align}
and
\begin{align}
\label{e13-main-t}
\cG_{n_i}(x) = \sum_{\ell = 0}^{q^h - 2} \epsilon_{\ell}(n_i)x^{\ell} = \sum_{\ell = 0}^{q^h - 2} \epsilon_{\ell}(z_{\bar{\beta}_i}(h))x^{\ell} = \Gamma(\bar{\beta}_i) \pmod{(x^{q^h - 1} - 1)}.
\end{align}

Since $\Gamma$ is a semigroup homomorphism, we deduce from (\ref{e11-main-t}) that
\begin{align*}
\Gamma(\bar{\alpha}) = \Gamma(\bar{\beta}_0) \cdots \Gamma(\bar{\beta}_r) \pmod{(x^{q^h - 1} - 1)},
\end{align*}
and it thus follows from (\ref{e12-main-t}) and (\ref{e13-main-t}) that
\begin{align}
\label{e14-main-t}
\cG_n(x) \equiv \cG_{n_0}(x)\cG_{n_1}(x)\cdots \cG_{n_r}(x) \pmod{(x^{q^h - 1} - 1)}.
\end{align}

Note that $\{n_0, \ldots, n_r\}$ is the set of all digits appearing in the $q^h$-adic expansion of $n$. Hence it follows from (\ref{e14-main-t}) and the definition of $\fc_j(n)$ that 
\begin{align*}
\cG_n(x) \equiv \prod_{j = 0}^{q^h - 1}\cG_j(x)^{\fc_j(n)} \pmod{(x^{q^h - 1} - 1)},
\end{align*}
which proves our contention.

\end{proof}

\subsection{Examples}

Let $\wp$ be a prime of degree $h$ in $\bA$, and let $\fq$ be a primitive root modulo $\wp$. The multiplicative group $(\bA/ \wp \bA)^{\times}$ consists of all power $\fq^i$ for $0 \le i \le q^h - 2$. Take an arbitrary integer $n \ge 0$. Theorem \ref{main-theorem} can tell us exactly the number of integers $m$ with $0 \le m \le n$ for which the Carlitz binomial coefficients $\binom{n}{m}_C$ fall into each of the residue classes $\fq^0, \fq^1, \ldots, \fq^{q^h - 2} \pmod{\wp}$. For each $0 \le j \le q^h - 2$, the number of integers $m$ with $0 \le m \le n$ for which the Carlitz binomial coefficients $\binom{n}{m}_C$ fall into the residue classes $\fq^j \pmod{\wp}$, is exactly the coefficient $\epsilon_j(n)$ of the polynomial $\cG_n(x)$ in Theorem \ref{main-theorem}. We will show how to compute the polynomial $\cG_n(x) \in \bZ[x]$. 

Let $\cP(x) \in \bZ[x]$ be the polynomial defined by
\begin{align}
\label{P(x)-example}
\cP(x) = \prod_{j = 0}^{q^h - 1}\cG_j(x)^{\fc_j(n)},
\end{align}
where the integers $\fc_j(n)$ are defined as in Theorem \ref{main-theorem}. By the Euclidean algorithm, there exists a unique couple $(\cQ(x), \cR(x)) \in \bZ[x] \times \bZ[x]$ such that
\begin{align}
\label{e1-example}
\cP(x) = (x^{q^h - 1} - 1)\cQ(x) + \cR(x)
\end{align}
and either $\cR(x) = 0$ or $\deg(\cR(x)) \le q^h - 2$. We know that $\deg(\cG_n(x)) \le q^h - 2$, and it thus follows from Theorem \ref{main-theorem} and (\ref{e1-example}) that 
\begin{align*}
\cG_n(x) = \cR(x).
\end{align*}
Hence in order to compute $\cG_n(x)$ for any integer $n \ge 0$, it suffices to compute $\cP(x)$, which in turn is equivalent to computing only the polynomials $\cG_j(x)$ with $\fc_j(n) \ne 0$, and then find the remainder of the division $\cP(x)/(x^{q^h - 1} - 1)$.

For illustration, we now give one explicit example. We maintain the same notation as in Theorem \ref{main-theorem}.  

Let $q = 3$, and let $\bA = \bF_3[T]$. Let $\wp = T^2 + 1$ be a prime in $\bA$ of degree $2$. So $h = 2$. Let $\fq = T + 1 \in \bA$. Then $\fq$ is a primitive root modulo $\wp$. Note that $\bA/\wp \bA)^{\times}$ is of cardinality $8$. We have
\begin{align*}
\fq^0 \pmod{\wp} &= 1, \\
\fq^1 \pmod{\wp} &= T + 1,\\
\fq^2 \pmod{\wp} &= 2T,\\
\fq^3 \pmod{\wp} &= 2T + 1,\\
\fq^4 \pmod{\wp} &= 2,\\
\fq^5 \pmod{\wp} &= 2T + 2,\\
\fq^6 \pmod{\wp} &= T,\\
\fq^7 \pmod{\wp} &= T + 2.
\end{align*}

Let $n = 1811$. Then 
\begin{align*}
n = 1811 =  2 + 3 \cdot (3^2) + 4\cdot (3^{2 \cdot 2}) + 2 \cdot (3^{2\cdot 3})
\end{align*}
is the $3^2$-adic expansion of $n$. Hence $\fc_2(n) = 2, \fc_3(n) = 1, \fc_4(n) = 1$, and $\fc_j(n) = 0$ for all $j \in \{0, 1, 5, 6, 7, 8\}$. It suffices to compute $\cG_2(x), \cG_3(x), \cG_4(x)$. 

We see that
\begin{align*}
1!_C &= 2!_C = 1, \\
3!_C &= 4!_C = T^3 - T.
\end{align*}
Thus (\ref{eq-cG_n}) implies that
\begin{align*}
\cG_2(x) &= 3, \\
\cG_3(x) &= 2 + 2x^6, \\
\cG_4(x) &= 4 + x^6.
\end{align*}
By (\ref{P(x)-example}), we deduce that
\begin{align*}
\cP(x) = \cG_2(x)^{\fc_2(n)} \cG_3(x)^{\fc_3(n)} \cG_4(x)^{\fc_4(n)} = 3^2(2 + 2x^6)(4 + x^6).
\end{align*}

We can write $\cP(x)$ in the form
\begin{align*}
\cP(x) = (x^8 - 1)(18x^4) + 90x^6 + 18x^4 + 72,
\end{align*}
and thus 
\begin{align}
\label{Galois-polynomial}
\cG_n(x) = \cG_{1811}(x) = 72 + 18x^4 + 90x^6.
\end{align}
Therefore $\epsilon_0(1811) = 72, \epsilon_4(1811) = 18, \epsilon_6(1811) = 90$, and $\epsilon_j(1811) = 0$ for all $j \in \{1, 2, 3, 5, 7\}$.

The coefficients of $\cG_{1811}(x)$ can tell us about the distribution of the Carlitz binomial coefficients $\binom{1811}{m}_C$ modulo $\wp$. For example, $\epsilon_6(1811) = 90$ is the number of the Carlitz binomial coefficients $\binom{1811}{m}_C$ with $0 \le m \le 1811$ that fall into the residue class $\fq^6 \pmod{\wp}$, i.e., there are exactly 90 integers $m$ with $0 \le m \le 1811$ such that
\begin{align*}
 \binom{1811}{m}_C \equiv T \pmod{T^2 + 1}.
\end{align*}

\end{document}